\documentclass{article}
\usepackage{amssymb,amsfonts,amsmath,amsthm,amsopn,amstext,amscd,latexsym,xy,mathdots,stmaryrd,xcolor} 
\usepackage{hyperref} \theoremstyle{plain} 
\input xy
\xyoption{all} 
 \newtheorem{theorem}{Theorem}[section]
\newtheorem{lemma}[theorem]{Lemma}
\newtheorem{proposition}[theorem]{Proposition}

 \theoremstyle{remark}
\newtheorem{remark}[theorem]{Remark} 
\numberwithin{equation}{section} \numberwithin{paragraph}{section}

\DeclareMathOperator{\Gal}{Gal}

\DeclareMathOperator{\ad}{Ad}

\DeclareMathOperator{\Frob}{Frob}

\newcommand{\cO}{{\mathcal O}}

\newcommand{\bbA}{{\mathbb A}}

\newcommand{\bbC}{{\mathbb C}}

\newcommand{\bbF}{{\mathbb F}}

\newcommand{\bbQ}{{\mathbb Q}}

\newcommand{\bbZ}{{\mathbb Z}}

\newcommand{\GL}{\mathrm{GL}}

\newcommand{\PGL}{\mathrm{PGL}}
\newcommand{\Proj}{\mathrm{Proj}}

\newcommand{\SL}{\mathrm{SL}}

\newcommand{\PSL}{\mathrm{PSL}}

\newcommand{\A}{\mathbb{A}}

\def\rhobar{ {\overline{\rho}} }

\newcommand{\ra}{\rightarrow}

\newcommand{\pdet}{\Delta}

\newcommand{\F}{{\mathbb F}}

\title{Modularity    of  $\PGL_2(\F_p)$-representations over totally real fields}

\author{Patrick B. Allen,\footnote{\textsc{Department of Mathematics and Statistics, McGill University, Montreal, Canada.} \textit{Email address}: \texttt{patrick.allen@mcgill.ca}}  \ \  Chandrashekhar B.  Khare\footnote{\textsc{Department of Mathematics, UCLA, Los Angeles, USA.} \textit{Email address}: \texttt{shekhar@math.ucla.edu}}  \ \ 
and  \ \ Jack A. Thorne\footnote{\textsc{Department of Pure Mathematics and Mathematical Statistics, Wilberforce Road, Cambridge, United Kingdom.} \textit{Email address:} \texttt{thorne@dpmms.cam.ac.uk}}} 

\begin{document}
\maketitle

\begin{abstract}

We study an analogue of Serre's modularity conjecture for projective representations $\rhobar: \Gal(\overline{K} / K) \ra \PGL_2(k)$, where $K$ is a totally real number field. We prove new cases of this conjecture when $k = \mathbb{F}_5$ by using the automorphy lifting theorems over CM fields established in \cite{AKT}.

\end{abstract}

\tableofcontents

\section{Introduction}

Let $K$ be a number field, and consider a continuous representation
\[ \rho : G_K \to \GL_2(k), \]
where $k$ is a finite field. (Here $G_K$ denotes the absolute Galois group of $K$; for this and other notation, see \S \ref{subsec_notation} below.) We say that $\rho$ is of  {\it Serre-type}, or $S$-type, if it is absolutely irreducible and totally odd, in the sense that for each real place $v$ of $K$ and each associated complex conjugation $c_v \in G_K$, $\det \rho(c_v) = -1$. 

Serre's conjecture and its generalisations assert that any $\rho$  of $S$-type should be automorphic (see for example \cite{Asterisque, Duke} in the case $K = \mathbb{Q}$, \cite{BDJ} when $K$ is totally real, and \cite{Sen18} for a general number field $ K$). The meaning of the word `automorphic' depends on the context but when $K$ is totally real, for example, we can ask for $\rho$ to be associated to a cuspidal automorphic representation $\pi$ of $\GL_2(\A_K)$ which is regular algebraic of weight 0 (see \S \ref{subsec_automorphy} below). Serre's conjecture is now a theorem when $K = \mathbb{Q}$ \cite{KW, Kha09}. For a totally real field $K$, some results are available when $k$ is `small'. These are summarised in the following theorem, which relies upon the papers \cite{Duke, Tunnell, SBT, Man, Ellenberg}:
\begin{theorem}\label{introthm_known_cases}
Let $K$ be a totally real number field, and let $\rho : G_K \to \GL_2(k)$ be a representation of $S$-type. Then $\rho$ is automorphic provided $| k | \in \{ 2, 3, 4, 5, 7, 9\}$.
\end{theorem}
One can equally consider continuous representations
\[ \sigma : G_K \to \PGL_2(k), \]
where again $k$ is a finite field. We say that $\sigma$ is of $S$-type if it is absolutely irreducible and totally odd, in the sense that if $k$ has odd characteristic then for each real place $v$ of $K$, $\sigma(c_v)$ is non-trivial. One could formulate a projective analogue of Serre's conjecture, asking that any representation $\sigma$ of $S$-type be automorphic. A theorem of Tate implies that $\sigma$  lifts to a linear  representation valued in $\GL_2(k')$ for some finite extension $k' / k$,  and by $\sigma$ being automorphic we mean that a lift of  it to a linear representation is automorphic (see \S \ref{subsec_automorphy} below).  Thus if $k$ is allowed to vary, this conjecture is equivalent to Serre's conjecture, since any representation $\rho$ has an associated projective representation $\Proj(\rho)$, and any projective representation $\sigma$  lifts to a representation valued in $\GL_2(k')$ for some finite extension $k' / k$; moreover, $\rho$ is of $S$-type if and only if $\Proj(\rho)$ is, and $\rho$ is automorphic if and only if $\Proj(\rho)$ is.

However, for fixed $k$ the two conjectures are not equivalent: certainly if $\rho$ is valued in $\GL_2(k)$ then $\Proj(\rho)$ takes values in $\PGL_2(k)$, but it is not true that any representation $\sigma : G_K \to \PGL_2(k)$ admits a lift valued in $\GL_2(k)$, and in fact in general the determination of the minimal extension $k' / k$ such that there is a lift to $\GL_2(k')$ is somewhat subtle. It is therefore of interest to ask whether the consideration of projective representations allows one to expand the list of `known' cases of Serre's conjecture.

Our main theorem affirms that this is indeed the case. Before giving the statement we need to introduce one more piece of notation. We write $\pdet : \PGL_2(k) \to k^\times / (k^\times)^2$ for the homomorphism induced by the determinant. We say that a homomorphism $G_K \to k^\times / (k^\times)^2$ is totally even (resp. totally odd) if each complex conjugation in $G_K$ is trivial (resp. non-trivial) image. 
\begin{theorem}\label{introthm_projective_cases}
Let $K$ be a totally real number field, and let $\sigma : G_K \to \PGL_2(k)$ be a representation of $S$-type. Then $\sigma$ is automorphic provided that one of the following conditions is satisfied: 
\begin{enumerate}
\item $|k| \in \{ 2, 3, 4 \}$.
\item $|k| = 5$, $[K(\zeta_5) : K] = 4$, and $\pdet \circ \sigma$ is totally even.
\item $|k| = 5$, $[K(\zeta_5) : K] = 4$, and $\pdet \circ \sigma$ is totally odd.
\item $|k| = 7$ and $\pdet \circ \sigma$ is totally odd.
\item $|k| = 9$ and $\pdet \circ \sigma$ is totally even. 
\end{enumerate}
\end{theorem}
 We note  the exceptional isomophisms $\PSL_2(\F_9)=A_6$, $\PGL_2(\F_5)=S_5$, $\PGL_2(\F_3)=S_4$, $\PGL_2(\F_2)=S_3$ which link our results to showing that splitting fields of polynomials of small degree over $K$ arise automorphically. 

The proof of Theorem \ref{introthm_projective_cases} falls into three cases. The first is when $|k|$ is even or $k = \bbF_3$. When $|k|$ is even, the homomorphism $\GL_2(k) \to \PGL_2(k)$ splits, so we reduce easily to Theorem \ref{introthm_known_cases}. When $k = \bbF_3$, the homomorphism $\PGL_2(\bbZ[\sqrt{-2}]) \to \PGL_2(\bbF_3)$ splits and we can use the Langlands--Tunnell theorem \cite{Tunnell} to establish the automorphy of $\sigma$.

The second  case is when  $|k|$ is odd and $-1$ is a square in $k$ (resp. a non-square in $k$) and $\Delta \circ \sigma$ is totally even (resp.  totally odd). In this case we are able to construct the following data:
\begin{itemize}
\item A solvable totally real extension $L / K$ and a representation $\rhobar_1 : G_L \to \GL_2(k)$ such that $\Proj(\rhobar_1) = \sigma|_{G_L}$ (by showing that $L / K$ can be chosen to kill the Galois cohomological obstruction to lifting).
\item A representation $\rho_2 : G_K \to \GL_2(\overline{\bbQ}_p)$ such that $\Proj(\overline{\rho}_2)$ and $\sigma$ are conjugate in $\PGL_2(\overline{\bbF}_p)$ (by choosing an arbitrary lift of $\sigma$ to $\GL_2(\overline{\bbF}_p)$ and applying the Khare--Wintenberger method). 
\end{itemize}
We can then use Theorem \ref{introthm_known_cases} to verify the automorphy of $\rhobar_1$, hence the residual automorphy of $\rhobar_2|_{G_L}$. An automorphy lifting theorem then implies the automorphy of $\rho_2|_{G_L}$, hence $\rho_2$ itself by solvable descent, hence finally of $\sigma$.

The final case is when $k = \bbF_5$ and $\Delta \circ \sigma$ is totally odd. In this case there does not exist any totally real extension $L / K$ such that $\sigma|_{G_L}$ lifts to a representation valued in $\GL_2(k)$ (there is a local obstruction at the real places). However, it is possible to find a CM extension $L / K$ such that $\sigma|_{G_L}$ lifts to a representation valued in $\GL_2(k)$ with determinant the cyclotomic character. When $k = \bbF_5$ such a representation necessarily appears in the group of  5-torsion points of an elliptic curve over $L$ (cf. \cite{SBT}) and so we can use the automorphy results over CM fields established in \cite{AKT} together with a solvable descent argument to obtain the automorphy of $\sigma$.  The main novelty in this paper is contained in our treatment of this case.
\begin{remark}\label{rmk_tunnell}
In the final case above of a representation $\sigma : G_K \to \PGL_2(\bbF_5)$ with non-solvable image,  the residual automorphy of the lift $\rho : G_L \to \GL_2(\bbF_5)$ ultimately depends on \cite[Theorem 7.1]{AKT}, which proves the automorphy of certain residually dihedral 2-adic Galois representations. The residual automorphy of these 2-adic representations is verified using automorphic induction. In particular, our proof in this case does not depend on the use of the Langlands--Tunnell theorem. This is in contrast to the argument used in e.g.\ \cite[Theorem 4.1]{SBT} to establish the automorphy of representations $\rho' : G_K \to \GL_2(\bbF_5)$ with cyclotomic determinant.

This `2-3 switch' strategy can also be used to prove the automorphy of representations $\sigma : G_K \to \PGL_2(\bbF_3)$ with $\pdet \circ \sigma$ totally odd  using the 2-adic automorphy theorems proved in \cite{Allen}, see  Theorem \ref{LT} of the text. This class of representations includes the projective representations associated to the Galois action on the 3-torsion points of an elliptic curve over $K$. This gives a way to verify the modulo 3 residual automorphy of elliptic curves over $K$ which does not rely on the Langlands--Tunnell theorem (and in particular the works \cite{L, Jac81}) but only on the Saito--Shintani lifting for holomorphic Hilbert modular forms \cite{Sai75}.  (We note that we do need to use the Langlands-Tunnell theorem  to prove the automorphy of representations $\sigma : G_K \to \PGL_2(\bbF_3)$ with $\pdet \circ \sigma$ totally even, cf. Theorem \ref{known}.)
\end{remark}

We now describe the structure of this note. We begin in \S \ref{sec_lifting} by studying the lifts of projective representations and collecting various results about the existence of characteristic 0 lifts of residual representations and their automorphy. We are then able to give the proofs of Theorem \ref{introthm_known_cases} and the first two cases in the proof of Theorem \ref{introthm_projective_cases} described above. In \S \ref{sec_mod_3}, we expand on Remark \ref{rmk_tunnell} by showing how the main theorems of \cite{Allen} can be used to give another proof of the automorphy of $S$-type representations $\sigma : G_K \to \PGL_2(\bbF_3)$ (still under the hypothesis that $K$ is totally real and $\Delta \circ \sigma$ is totally odd). Finally, in \S \ref{sec_mod_5} we use similar arguments, now based on the main theorems of \cite{AKT}, to complete the proof of Theorem \ref{introthm_projective_cases}.

\subsection*{Acknowledgments} 

We would like to thank the anonymous referee for their comments and corrections. 
P.A. was supported by Simons Foundation Collaboration Grant 527275 and NSF grant DMS-1902155. He would like to thank the third author and Cambridge University for hospitality during a visit where some of this work was completed. Parts of this work were completed while P.A. was a visitor at the Institute for Advanced Study, where he was partially supported by the NSF. He would like to thank the IAS for providing excellent working conditions during his stay.  J.T.'s work received funding from the European Research Council (ERC) under the European Union's Horizon 2020 research and innovation programme (grant agreement No 714405).  This research was begun during the period that J.T. served as a Clay Research Fellow.

\subsection{Notation}\label{subsec_notation}

If $K$ is a perfect field then we write $G_K = \Gal(\overline{K} / K)$ for the Galois group of $K$ with respect to a fixed choice of algebraic closure. If $K$ is a number field and $v$ is a place of $K$ then we write $K_v$ for the completion of $K$ at $v$, and fix an embedding $\overline{K} \to \overline{K}_v$ extending the natural embedding $K \to K_v$; this determines an injective homomorphism $G_{K_v} \to G_K$. If $v$ is a finite place of $K$ then we write $\Frob_v \in G_{K_v}$ for a lift of the geometric Frobenius, $k(v)$ for the residue field of $K_v$, and $q_v$ for the cardinality of $K_v$; if $v$ is a real place, then we write $c_v \in G_{K_v}$ for complex conjugation. Any homomorphism from a Galois group $G_K$ to another topological group will be assumed to be continuous.

If $p$ is a prime and $K$ is a field of characteristic 0, then we write $\epsilon : G_K \to \bbZ_p^\times$ for the $p$-adic cyclotomic character, $\overline{\epsilon} : G_K \to \bbF_p^\times$ for its reduction modulo $p$, and $\omega : G_K \to  \bbF_p^\times / ( \bbF_p^\times)^2$ for the character $\overline{\epsilon} \text{ mod } ( \bbF_p^\times)^2$. More generally, if $\rho : G_K \to \GL_n(\overline{\bbQ}_p)$ is a representation, then we write $\rhobar : G_K \to \GL_n(\overline{\bbF}_p)$ for the associated semisimple residual representation (uniquely determined up to conjugation).

If $k$ is a field then we write $\Proj : \GL_n(k) \to \PGL_n(k)$ for the natural projection and $\pdet : \PGL_n(k) \to k^\times / (k^\times)^n$ for the character induced by the determinant. We will only use these maps in the case $n = 2$.

If $K$ is a field of characteristic 0, $E$ is an elliptic curve curve over $K$, and $p$ is a prime, then we write $\rhobar_{E, p} : G_K \to \GL_2(\bbF_p)$ for the representation associated to $H^1(E_{\overline{K}}, \bbF_p)$ after a choice of basis. Thus $\det \rhobar_{E, p} = \overline{\epsilon}^{-1}$.

\section{Lifting  representations}\label{sec_lifting}

In this section we study different kinds of liftings of representations: liftings to characteristic 0 (and the automorphy of such liftings) and liftings of projective representations to true (linear) representations. We begin by discussing what it means for a (projective or linear) representation to be automorphic.

\subsection{Automorphy of linear and projective representations}\label{subsec_automorphy}

Let $K$ be a CM or totally real number field. If $\pi$ is a cuspidal, regular algebraic automorphic representation of $\GL_2(\A_K)$ then (see e.g. \cite{Tay89, hltt}) for any isomorphism $\iota : \overline{\bbQ}_p \to \bbC$, there exists a semisimple representation $r_\iota(\pi) : G_K \to \GL_2(\overline{\bbQ}_p)$ satisfying the following condition, which determines $r_\iota(\pi)$ uniquely up to conjugation: for all but finitely many finite places $v$ of $K$ such that $\pi_v$ is unramified, $r_\iota(\pi)|_{G_{K_v}}$ is unramified and $r_\iota(\pi)|_{G_{K_v}}^{ss}$ is related to the representation $\iota^{-1}\pi_v$ under the Tate-normalised unramified local Langlands correspondence. (See \cite[\S 2]{Tho16} for an explanation of how the characterstic polynomial of $r_\iota(\pi)|_{G_{K_v}}$ may be expressed in terms of the eigenvalues explicit unramified Hecke operators.) In this paper we only need to consider automorphic representations which are of regular algebraic automorphic representations $\pi$ which are of weight 0, in the sense that for each place $v | \infty$ of $K$, $\pi_v$ has the same infinitesimal character as the trivial representation.

Let $k$ be a finite field of characteristic $p$, viewed inside its algebraic closure $\overline{\bbF}_p$. In this paper, we say that a representation $\rho : G_K \to \GL_2(k)$ is automorphic if it is $\GL_2(\overline{\bbF}_p)$-conjugate to a representation of the form $\overline{r_\iota(\pi)}$, where $\pi$ is a cuspidal, regular algebraic automorphic representation of $\GL_2(\bbA_K)$ of weight 0. We say that a representation $\sigma : G_K \to \PGL_2(k)$ is automorphic if it is $\PGL_2(\overline{\bbF}_p)$-conjugate to a representation of the form $\Proj(\overline{r_\iota(\pi)})$, where $\pi$ is a cuspidal, regular algebraic automorphic representation of $\GL_2(\bbA_K)$ of weight 0. 

We say that a representation $\rho : G_K \to \GL_2(\overline{\bbQ}_p)$ is automorphic if it is conjugate to a representation of the form $\overline{r_\iota(\pi)}$, where $\pi$ is a cuspidal, regular algebraic automorphic representation of $\GL_2(\bbA_K)$ of weight 0. We say that an elliptic curve $E$ over $K$ is modular if the representation of $G_K$ afforded by $H^1(E_{\overline{K}}, \bbQ_p)$ is automorphic in this sense. 
\begin{lemma}
Let $K$ be a CM or totally real number field, let $\rho : G_K \to \GL_2(k)$ be a representation, and let $\sigma = \Proj(\rho)$. Then:
\begin{enumerate}
\item Let $\chi : G_K \to k^\times$ be a character. Then $\rho$ is automorphic if and only if $\rho \otimes \chi$ is automorphic.
\item $\sigma$ is automorphic if and only if $\rho$ is automorphic.
\end{enumerate}
\end{lemma}
\begin{proof}
If $\chi : G_K \to k^\times$ is a character then its Teichm\"uller lift $X : G_K \to \overline{\bbQ}_p^\times$ is associated, by class field theory, to a finite order Hecke character $\Xi : \bbA_K^\times \to \bbC^\times$. If $\pi$ is a cuspidal automorphic representation which is regular algebraic of weight 0 and $\overline{r_\iota(\pi)}$ is conjugate to $\rho$, then $\pi \otimes (\Xi \circ \det)$ is also cuspidal and regular algebraic of weight 0 and $\overline{r_\iota(\pi \otimes (\Xi \circ \det))}$ is conjugate to $\rho \otimes \chi$.

It is clear from the definition that if $\rho$ is automorphic  then so is $\sigma$. Conversely, if $\sigma$ is automorphic then there is a cuspidal, regular algebraic automorphic representation $\pi$ of $\GL_2(\A_K)$ and isomorphism $\iota : \overline{\bbQ}_p \to \bbC$ such that $\Proj(\overline{r_\iota(\pi})) = \Proj(\rho)$. It follows that there exists a character $\chi : G_K \to \overline{\bbF}_p^\times$ such that $\rho$ is conjugate to $\overline{r_\iota(\pi)} \otimes \chi$. The automorphy of $\rho$ follows from the first part of the lemma.
\end{proof}
\subsection{Lifting to characteristic 0}

We recall a result on the existence of liftings with prescribed properties. We first need to say what it means for a representation to be exceptional. If $K$ is a number field and $\sigma : G_K \to \PGL_2(k)$ is a projective representation, we say that $\sigma$ is exceptional if it is $\PGL_2(\overline{\bbF}_p)$-conjugate to a representation $\sigma' : G_K \to \PGL_2(\bbF_5)$ such that $\sigma'(G_K)$ contains $\PSL_2(\bbF_5)$ and the character $(-1)^{\Delta \circ \sigma'} \overline{\epsilon}$ is trivial. (Here we write $(-1)^{\Delta \circ \sigma'} $ for the composition of $\Delta \circ \sigma'$ with the unique isomorphism $\bbF_5^\times / (\bbF_5^\times)^2 \cong \{ \pm 1 \}$.) We say that a representation $\rho : G_K \to \GL_2(k)$ is exceptional if $\Proj(\rho)$ is exceptional. If $K$ is totally real then this is equivalent to the definition given in \cite[\S 3]{khare-thorne-mathz17}. The exceptional case is often excluded in the statements of automorphy lifting theorems (the root cause being the non-triviality of the group $H^1(\sigma(G_K), \ad^0 \rho(1))$).
\begin{theorem}\label{thm:geo-lift}
Let $K$ be a totally real field, let $\rhobar : G_K \ra \GL_2(k)$ be a representation of $S$-type, and let $\psi : G_K \to \overline{\bbZ}_p^\times$ be a continuous character lifting $\det \rhobar$ such that $\psi\epsilon$ is of finite order.  Suppose that the following conditions are satisfied:
\begin{enumerate}
	\item $p > 2$ and  $\rhobar|_{G_{K(\zeta_p)}}$ is absolutely irreducible. 
	\item If $p = 5$ then $\rhobar$ is non-exceptional.
\end{enumerate}
 Then $\rhobar$ lifts to a continuous representation $\rho : G_K \to \GL_2(\overline{\bbZ}_p)$ satisfying the following conditions:
\begin{enumerate}
\item For all but finitely many places $v$ of $F$, $\rho|_{G_{K_v}}$ is unramified.
\item $\det \rho = \psi$.
\item For each place $v | p$ of $K$, $\rho|_{G_{K_v}}$ is potentially crystalline and for each embedding $\tau : K_v \to \overline{\bbQ}_p$, $\mathrm{HT}_\tau(\rho) = \{ 0, 1 \}$. 
Moreover, for any $v | p$ such that $\rhobar|_{G_{K_v}}$ is reducible, we can assume that $\rho|_{G_{K_v}}$ is ordinary, in the sense of \cite[\S 5.1]{Tho16}.
\end{enumerate}
\end{theorem}
\begin{proof}
	This follows from \cite[Theorem 7.6.1]{snowden2009dimensional}, on noting that the condition (A2) there can be replaced by the more general condition that $\rhobar$ is non-exceptional (indeed, the condition (A2) is used to invoke \cite[Proposition 3.2.5]{Kis09a}, which is proved under this more general condition). To verify the existence of a potentially crystalline lift of $\rhobar|_{G_{K_v}}$ for each $v | p$ (or in the terminology of \emph{loc. cit.}, the compatibility of $\rhobar|_{G_{K_v}}$ with type $A$ or $B$) we apply \cite[Proposition 7.8.1]{snowden2009dimensional} (when $\rhobar|_{G_{K_v}}$ is irreducible) or \cite[Lemma 6.1.6]{Bar12b} (when $\rhobar|_{G_{K_v}}$ is reducible).
\end{proof}
We next recall an automorphy lifting theorem.
\begin{theorem}\label{thm_ALT}
Let $K$ be a totally real number field, and let $\rho : G_K \to \GL_2(\overline{\bbZ}_p)$ be a continuous representation satisfying the following conditions:
\begin{enumerate}
\item $p > 2$ and $\rhobar|_{G_{K(\zeta_p)}}$ is absolutely irreducible.
\item For all but finitely many finite places $v$ of $K$, $\rho|_{G_{K_v}}$ is unramified.
\item For each place $v | p$ of $K$, $\rho|_{G_{K_v}}$ is de Rham and for each embedding $\tau : K_v \to \overline{\bbQ}_p$, $\mathrm{HT}_\tau(\rho)  = \{ 0, 1 \}$.
\item The representation $\rhobar$ is automorphic.
\end{enumerate}
Then $\rho$ is automorphic.
\end{theorem}
\begin{proof}
This follows from  \cite[Theorem 9.3]{khare-thorne-mathz17}.
\end{proof}
We now combine the previous two theorems to obtain a ``solvable descent of automorphy'' theorem for residual representations, along similar lines to \cite{K, Tay}.
\begin{proposition}\label{realSolvableDescent}
Let $K$ be a totally real number field and let $\rhobar : G_K \ra \GL_2(k)$ be a representation of $S$-type. Suppose that there exists a solvable totally real extension $L / K$ such that the following conditions are satisfied:
	\begin{enumerate}
		\item $p > 2$ and $\rhobar|_{G_{L(\zeta_p)}}$ is absolutely irreducible. If $p = 5$, then $\rhobar$ is non-exceptional.
		\item $\rhobar|_{G_L}$ is automorphic.
	\end{enumerate}
Then $\rhobar$ is automorphic.
\end{proposition}
\begin{proof}
	Let $\psi : G_K \to \overline{\bbZ}_p^\times$ be the character such that $\psi \epsilon$ is the Teichm\"uller lift of $(\det \rhobar) \overline{\epsilon}$, and let $\rho : G_K  \to \GL_2(\overline{\bbZ}_p)$ be the lift of $\rho$ whose existence is asserted by Theorem~\ref{thm:geo-lift}. Then Theorem \ref{thm_ALT} implies the automorphy of $\rho|_{G_L}$, and the automorphy of $\rho$ itself and hence of $\rhobar$ follows by cyclic descent, using the results of Langlands \cite{L}.
\end{proof}
We can now give the proof of Theorem \ref{introthm_known_cases}, which we restate here for the convenience of the reader:
\begin{theorem}\label{thm-GL2}
Let $K$ be a totally real field and let $\rhobar:G_K \ra \GL_2(k)$ be a representation of $S$-type.
Suppose that $\lvert k \rvert \in \{2, 3, 4, 5, 7, 9\}$. Then $\rhobar$ is automorphic.
\end{theorem}

\begin{proof}
Many of the results we quote here are stated in the case of $K = \bbQ$ but hold more generally for totally real fields with minor modification. 
We will apply them in the more general setting without further comment.

If $\rhobar$ is dihedral then this is a consequence of results of Hecke (see \cite[\S5.1]{Duke}).
If $k = \bbF_3$, it is a consequence of the Langlands--Tunnell theorem \cite{Tunnell} (see the discussion following Theorem~5.1 in \cite[Chapter~5]{Wiles}). 
We may thus assume for the remainder of the proof that $\lvert k \rvert > 3$. 
We may also assume that for any abelian extension $L/K$, the restriction $\rhobar|_{G_{L(\zeta_p)}}$ is absolutely irreducible (as otherwise $\rhobar$ would be dihedral).

Next suppose that $k = \bbF_5$. We note that $\rhobar$ is not exceptional, by \cite[Lemma 3.1]{khare-thorne-mathz17}.  Let $L/K$ be the totally real cyclic extension cut out by $(\det\rhobar)\overline{\epsilon}$. 
By \cite[Theorem~1.2]{SBT}, there is an elliptic curve $E$ over $L$ such that $\rhobar_{E, 5} \cong \rhobar|_{G_L}$ and $\rhobar_{E, 3}(G_L)$ contains $\SL_2(\bbF_3)$. By the $k = \bbF_3$ case of the theorem and by Theorem \ref{thm_ALT}, we see that $E$ is automorphic, hence so is $\rhobar|_{G_{L}}$. 
The automorphy of $\rhobar$ then follows from Proposition~\ref{realSolvableDescent}.
The $k = \bbF_7$ case is similar, using \cite[Proposition~3.1]{Man} instead of \cite[Theorem~1.2]{SBT}.

Next suppose that $k = \bbF_4$. 
We can twist $\rho$ to assume that it is valued in $\SL_2(\bbF_4)$. 
Then \cite[Theorem~3.4]{SBT} shows that there is an abelian surface $A$ over $F$ with real multiplication by $\cO_{\bbQ(\sqrt{5})}$ such that the $G_K$-representation on $A[2] \cong \bbF_4^2$ is isomorphic to $\rhobar$ and such that the $G_K$-representation on $A[\sqrt{5}] \cong \bbF_5^2$ has image containing $\SL_2(\bbF_5)$. 
By the $k = \bbF_5$ case of the theorem, Theorem \ref{thm:geo-lift}, and Theorem \ref{thm_ALT}, we see that $A$ is automorphic, hence so is $\rhobar$.

Finally suppose that $k = \bbF_9$. 
Let $L/K$ be the totally real cyclic extension cut out by $(\det\rhobar)\overline{\epsilon}$. 
Then the argument of \cite[\S 2.5]{Ellenberg} shows that there is a solvable totally real extension $M / K$ containing $L / K$ and an abelian surface $A$ over $M$ with real multiplication by $\cO_{\bbQ(\sqrt{5})}$ such that the $G_{M}$-representation on $A[3] \cong \bbF_9^2$ is isomorphic to $\overline{\rho}|_{G_{M}} \otimes\overline{\epsilon}$ and such that the $G_{M}$-representation on $A[\sqrt{5}] \cong \bbF_5^2$ has image containing $\SL_2(\bbF_5)$. 
By the $k = \bbF_5$ case of the theorem, Theorem \ref{thm:geo-lift}, and Theorem \ref{thm_ALT}, we see that $A$ is automorphic, hence so is $\rhobar|_{G_{M}}$. 
The automorphy of $\rhobar$ follows from Proposition~\ref{realSolvableDescent}.
\end{proof}

\begin{remark}
The `2-3 switch’ strategy  employed in Theorem \ref{LT} below, can be used to prove automorphy of totally odd representations $\rho : G_K \to \GL_2(\bbF_3)$ without using the Langlands-Tunnell theorem.  

\end{remark}

\subsection{Lifting projective representations}\label{subsec_lifting_projective}

We now consider the problem of lifting projective representations. 
\begin{lemma}\label{lem_tate}
Let $K$ be a number field, and let $\sigma : G_K \to \PGL_2(k)$ be a continuous homomorphism. Then there exists a finite extension $k' / k$ such that $\sigma$ lifts to a homomorphism $\rho : G_K \to \GL_2(k')$.
\end{lemma}
\begin{proof}
The obstruction to lifting a continuous homomorphism $\sigma : G_K \to \PGL_2(k)$ to a continuous homomorphism $\rho : G_K \to \GL_2(\overline{k})$ lies in $H^2(G_K, \overline{k}^\times) $. Tate proved that $H^2(G_K, \overline{k}^\times) = 0$ (see \cite[\S 6.5]{Ser77}) so a lift always exists. 
\end{proof}
\begin{lemma}\label{lem_lifting_over_solvable_extension}
Suppose that $p > 2$, let $K$ be a number field, and let $\sigma : G_K \to \PGL_2(k)$ be a homomorphism. Let $S$ be a finite set of places of $K$ such that for each $v \in S$, there exists a lift of $\sigma|_{G_{K_v}}$ to a homomorphism $\rho_v : G_K \to \GL_2(k)$. Then we can find the following data:
\begin{enumerate}
\item A solvable $S$-split extension $L / K$.
\item A homomorphism $\rho : G_L \to \GL_2(k)$ such that $\Proj(\rho) = \sigma|_{G_L}$ and for each $v \in S$ and each place $w | v$ of $L$, $\rho|_{G_{L_w}} = \rho_v$.
\end{enumerate}
Moreover, if $K$ is a CM field we can choose $L$ also to be a CM field. 
\end{lemma}
\begin{proof}
Let $H$ denote the $2$-Sylow subgroup of $k^\times$, of order $2^m$, and let $H' \leq k^\times$ denote its prime-to-2 complement. If $0 \leq k \leq m$, we write $G_k = \GL_2(k) / (2^{m-k} H \times H')$, which is an extension
\[ 1 \to H / 2^{m-k} H \to G_k \to \PGL_2(k) \to 1. \]
We show by induction on $k \geq 0$ that we can find a solvable, $S$-split extension $L_k / K$ and a homomorphism $\rho_k : G_{L_k} \to G_k$ lifting $\sigma|_{G_{L_k}}$ and such that for each $v \in S$ and each place $w | v$ of $L_k$, $\rho_k|_{G_{L_{k, w}}} = \rho_v \text{ mod }2^{m-k} H \times H'$. The case $k = 0$ is the existence of $\sigma$. The case $k = m$ implies the statement of the lemma, since $\GL_2(k) = G_m \times H'$. (Note \cite[Ch. X, Theorem 5]{Art09} implies that any collection of characters $\chi_v : G_{K_v} \to H'$ can be globalised to a character $\chi : G_K \to H'$.)

For the induction step, suppose the induction hypothesis holds for a fixed value of $k$. We consider the obstruction to lifting $\rho_k$ to a homomorphism $\rho_{k+1} : G_{L_k} \to G_{k+1}$. This defines an element of $H^2(G_{L_k}, \bbZ / 2 \bbZ)$ which is locally trivial at the places of $L_k$ lying above $S$. We can therefore find an extension of the form $L_{k+1} = L_k \cdot E_{k+1}$, where $E_{k+1} / K$ is a solvable $S$-split extension, such that the image of this obstruction class in $H^2(G_{L_{k+1}}, \bbZ / 2 \bbZ)$ vanishes and so there is a homomorphism $\rho'_{k+1} : G_{L_{k+1}} \to G_{k+1}$ lifting $\rho_k|_{G_{L_{k+1}}}$. 

If $v \in S$ and $w | v$ is a place of $L_{k+1}$ then there is a character $\chi_w : G_{L_{k+1, w}} \to \bbZ / 2 \bbZ$ such that $\rho'_{k+1}|_{G_{L_{k+1, w}}} = (\rho_v \text{ mod }2^{m-(k+1)} H \times H')\cdot \chi_w$. We can certainly find a character $\chi : G_{L_{k+1}} \to \bbZ / 2 \bbZ$ such that $\chi|_{G_{L_{k+1, w}}} = \chi_w$ for each such place $w$. The induction step is complete on taking $\rho_{k+1} = \rho'_{k+1} \cdot \chi$. 

It remains to explain why we can choose $K$ to be CM if $L$ is. Since the extensions $E_k$ in the proof are required only to satisfy some local conditions, which are vacuous if $K$ is CM, we can choose the fields $E_k$ to be of the form $KE_k’$ where $E_k’$ is a totally real extension, in which case the field  $L$  constructed in the proof  is  seen to be CM.
\end{proof}

\begin{remark}\label{utility}
We remark that if $v$ is a real place of $K$ and $\sigma(c_v) \neq 1$, then there exists a lift of $\sigma|_{G_{K_v}}$ to $\GL_2(k)$ if and only if either $-1$ is a square in $k^\times$ and $\Delta \circ \sigma(c_v) = 1$, or $-1$ is not a square in $k^\times$ and $\Delta \circ \sigma(c_v) \neq 1$. We also note the utility of the `$S$-split' condition: we can add any set of places at which $\sigma$ is unramified to $S$, and in this way ensure that the $S$-split extension $L / K$ is linearly disjoint from any other fixed finite extension of $K$.
\end{remark}

Here is a variant.
\begin{lemma}\label{lem_liftings_with_prescribed_determinant}
Suppose that $p > 2$. Let $K$ be a number field, let $\sigma : G_K \to \PGL_2(k)$ be a homomorphism, and let $\chi : G_K \to k^\times$ be a character. Suppose that the following conditions are satisfied:
\begin{enumerate}
\item $\Delta \circ \sigma = \chi \text{ mod }(k^\times)^2$.
\item For each finite place $v$ of $K$, $\sigma|_{G_{K_v}}$ and $\chi|_{G_{K_v}}$ are unramified.
\item For each real place $v$ of $K$, $\sigma(c_v) \neq 1$ and $\chi(c_v) = -1$.
\end{enumerate}
Then there exists a homomorphism $\rho : G_K \to \GL_2(k)$ such that $\Proj(\rho) = \sigma$ and $\det(\rho) = \chi$.
\end{lemma}
\begin{proof}
We consider the short exact sequence of groups
\[ 1 \to \{ \pm 1 \} \to \GL_2(k) \to \PGL_2(k) \times_{\Delta} k^\times \to 1, \]
where the last group is the subgroup of $(g, \alpha) \in \PGL_2(k) \times k^\times$ such that $\Delta(g) = \alpha \text{ mod }(k^\times)^2$. By hypothesis the pair $(\sigma, \chi)$ defines a homomorphism $\Sigma : G_K \to \PGL_2(k) \times_\Delta k^\times$ such that for every place $v$ of $K$, $\Sigma|_{G_{K_v}}$ lifts to $\GL_2(k)$ (see Remark~\ref{utility}). The subgroup of locally trivial elements of $H^2(G_K, \{ \pm 1 \})$ is trivial, by class field theory, so $\Sigma$ lifts to a homomorphism $\rho : G_K \to \GL_2(k)$, as required.
\end{proof}
We now prove an analogue of Proposition \ref{realSolvableDescent} for projective representations.
\begin{proposition}\label{prop_projective_solvable_descent}
Let $K$ be a totally real number field and let $\sigma : G_K \to \PGL_2(k)$ be a representation of $S$-type. Suppose that there exists a solvable totally real extension $L / K$ satisfying the following conditions:
\begin{enumerate}
\item $p > 2$ and $\sigma|_{G_{L(\zeta_p)}}$ is absolutely irreducible. If $p = 5$, then $\sigma$ is non-exceptional.
\item $\sigma|_{G_L}$ is automorphic.
\end{enumerate}
Then $\sigma$ is automorphic.
\end{proposition}
\begin{proof}
By Lemma \ref{lem_tate}, we can lift $\sigma$ to a representation $\rhobar : G_K \to \GL_2(\overline{k})$. Then $\rhobar|_{G_L}$ is automorphic and we can apply Proposition \ref{realSolvableDescent} to conclude that $\rhobar$ is automorphic, hence that $\sigma$ is automorphic.
\end{proof}
We are now in a position to establish a large part of Theorem \ref{introthm_projective_cases}.
\begin{theorem}\label{known}
Let $K$ be a totally real number field and let $\sigma : G_K \to \PGL_2(k)$ be a representation of $S$-type. If one of the following conditions holds, then $\sigma$ is automorphic: 
\begin{enumerate}
\item $|k| \in \{ 2, 3, 4 \}$.
\item $|k| = 5$ or $9$ and $\pdet \circ \sigma$ is totally even. If $|k| = 5$, then $\sigma$ is non-exceptional.
\item $|k| = 7$ and $\pdet \circ \sigma$ is totally odd.
\end{enumerate}
\end{theorem}
\begin{proof}
When $k = \F_2$ or $\F_4$, the map $\SL_2(k) \to \PGL_2(k)$ is an isomorphism, so $\sigma$ trivially lifts to a $\GL_2(k)$ representation and we can apply Theorem~\ref{thm-GL2}. The case when $|k|=3$ follows from \cite{Tunnell}.  In the other cases, we can assume that $\sigma|_{G_{K(\zeta_p)}}$ is absolutely irreducible (as otherwise $\sigma$ lifts to a dihedral representation). 
Let $S_\infty$ be the set of infinite places of $K$ and choose a finite set $S'$ of finite places of $K$ at which $\sigma$ is unramified such that $\Gal(\overline{K}^{\ker(\sigma|_{G_{K(\zeta_p)}})}/K)$ is generated by $\{\Frob_v\}_{v \in S'}$. 
We can apply Lemma \ref{lem_lifting_over_solvable_extension}, see also Remark~\ref{utility}, with $S = S_\infty \cup S'$ to find a solvable, totally real extension $L / K$ such that $\sigma$ lifts to a representation $\rhobar : G_L \to \GL_2(k)$ such that $\rhobar|_{G_{L(\zeta_p)}}$ is absolutely irreducible and $\rhobar$ is not exceptional if $p = 5$. Then Theorem~\ref{thm-GL2} implies the automorphy of $\rhobar$ and Proposition \ref{prop_projective_solvable_descent} implies the automorphy of $\sigma$, as desired.
\end{proof}

\section{Modularity of mod $3$ representations}\label{sec_mod_3}

In this section, which is a warm-up for the next one, we give a proof of the following theorem that does not depend on the Langlands--Tunnell theorem:
\begin{theorem}\label{LT}
Let $K$ be a totally real number field, and let $\sigma : G_K \to \PGL_2(\bbF_3)$ be a representation of $S$-type such that $\Delta \circ \sigma$ is totally odd. Then $\sigma$ is automorphic.
\end{theorem}
\begin{proof}
We can assume that $\sigma$ is not dihedral; by the classification of finite subgroups of $\PGL_2(\bbF_3)$, we can therefore assume that $\sigma(G_K)$ contains $\PSL_2(\bbF_3)$. By Proposition \ref{prop_projective_solvable_descent}, we can moreover assume, after replacing $K$ by a solvable totally real extension, that $\sigma$ is everywhere unramified and that for each place $v | 2$ of $K$, $q_v \equiv 1 \text{ mod }3$ and $\sigma|_{G_{K_v}}$ is trivial.
\begin{lemma}
There exists a solvable totally real extension $L / K$ and a modular elliptic curve $E$ over $L$ satisfying the following conditions:
\begin{enumerate}
\item $\sigma(G_L)$ contains $\PSL_2(\bbF_3)$. In particular, $\sigma|_{G_L}$ is of $S$-type.
\item The homomorphism $\Proj(\rhobar_{E, 3})$ is $\PGL_2(\overline{\bbF}_3)$-conjugate to $\sigma|_{G_L}$.
\end{enumerate}
\end{lemma}
\begin{proof}
The character $(\Delta \circ \sigma) \omega : G_K \to \bbF_3^\times$ is totally even, so cuts out a totally real (trivial or quadratic) extension $L / K$, and $\sigma(G_{L})$ contains $\PSL_2(\bbF_3)$ and satisfies $\Delta \circ \sigma|_{G_{L}} = \omega$. Using Lemma \ref{lem_liftings_with_prescribed_determinant}, we can find a lift $\rhobar : G_L \to \GL_2(\bbF_3)$ of $\sigma|_{G_L}$ satisfying the following conditions:
\begin{itemize}
\item $\det \rhobar = \overline{\epsilon}^{-1}$.
\item For each place $v | 2$ of $L$, $\rhobar|_{G_L}$ is trivial.
\item $\rhobar(G_L)$ contains $\SL_2(\bbF_3)$. In particular, $\rhobar|_{G_{L(\zeta_3)}}$ is absolutely irreducible.
\end{itemize}
We can then apply \cite[Lemma 9.7]{AKT} to conclude that there exists an elliptic curve $E / L$ satisfying the following conditions:
\begin{itemize}
\item There is an isomorphism $\rhobar_{E, 3} \cong \rhobar$.
\item For each place $v | 2$ of $L$, $E$ has multiplicative reduction at $v$ and the valuation at $v$ of the minimal discriminant of $E$ is 3.
\item $\rhobar_{E, 2}(G_L) = \SL_2(\bbF_2)$.
\end{itemize}
Then \cite[p. 1237, Corollary]{Allen} implies that $E$ is modular, proving the lemma.
\end{proof}
We see that $\sigma|_{G_L}$ is automorphic. We can then apply Proposition \ref{prop_projective_solvable_descent} to conclude that $\sigma$ itself is automorphic, as required.
\end{proof}

\section{Modularity of mod $5$  representations}\label{sec_mod_5}

In this section we complete the proof of Theorem \ref{introthm_projective_cases} by proving Theorem \ref{thm_mod_5} below. 
\begin{theorem}\label{thm_mod_5}
Let $K$ be a totally real field, and let $\sigma : G_K \to \PGL_2(\bbF_5)$ be a representation of $S$-type which is non-exceptional, and such that $\Delta \circ \sigma$ is totally odd. Then $\sigma$ is automorphic.
\end{theorem}
\begin{proof}
By the classification of subgroups of $\PGL_2(\bbF_5)$, automorphic induction, and the Langlands--Tunnell theorem, we can assume that $\sigma(G_K)$ contains $\PSL_2(\bbF_5)$. By Theorem \ref{thm:geo-lift}, Lemma \ref{lem_tate}, and Proposition \ref{prop_projective_solvable_descent} we can assume, after possibly replacing $K$ be a solvable totally real extension, that the following conditions are satisfied:
\begin{itemize}
\item There exists a representation $\rhobar : G_K \to \GL_2(\overline{\bbF}_5)$ such that $\Proj(\rhobar)$ is $\PGL_2(\overline{\bbF}_5)$-conjugate to $\sigma$. Moreover, $\rhobar$ is everywhere unramified.
\item There exists a representation $\rho : G_K \to \GL_2(\overline{\bbQ}_5)$ lifting $\rhobar$, which is unramified almost everywhere.
\item For each place $v | 5$ of $K$, $\zeta_5 \in K_v$, $\rhobar|_{G_{K_v}}$ is trivial, and $\rho|_{G_{K_v}}$ is ordinary, in the sense of \cite[\S 5.1]{Tho16}.
\item Let $\chi = \det \rho$. Then $\chi \epsilon$ has finite order prime to $5$ and for each finite place $v$ of $K$, $\chi \epsilon|_{G_{K_v}}$ is unramified. In particular, $\overline{\chi}$ is everywhere unramified.
\end{itemize}
Let $K' / K$ denote the quadratic CM extension cut out by the character $(\Delta \circ \sigma) \omega$.
\begin{lemma}
The representation $\rhobar|_{G_{K'}}$ is decomposed generic in the sense of \cite[Definition 4.3.1]{10authors}.
\end{lemma}
\begin{proof}
It is enough to find a prime number $l$ such that $l$ splits in $K'$ and for each place $v | l$ of $K'$, $q_v \equiv 1 \text{ mod }5$ and the eigenvalues of $\rhobar(\Frob_v)$ are distinct.  The argument of \cite[Lemma 7.1.5, (3)]{10authors} will imply  the existence of such a prime $l$ if we can show that if $M = K'(\zeta_5)$ and $\widetilde{M} / \bbQ$ is the Galois closure of $M/\bbQ$, then $\sigma(G_{\widetilde{M}})$ contains $\PSL_2(\bbF_5)$. To see this, first let $\widetilde{K} / \bbQ$ be the Galois closure of $K / \bbQ$. Then $\widetilde{K}$ is totally real, and so $\sigma(G_{\widetilde{K}}) = \sigma(G_K) = \PGL_2(\bbF_5)$ because $\Delta \circ \sigma$ is totally odd. The extension $M \widetilde{K} / \widetilde{K}$ is abelian, so $\widetilde{M} / \widetilde{K}$ is abelian and $\sigma(G_{\widetilde{M}})$ must contain $\PSL_2(\bbF_5)$.
\end{proof}
By construction, $\Delta \circ \sigma|_{G_{K'}} = \overline{\epsilon}^{-1} \text{ mod }(\bbF_5^\times)^2$, so by Lemma \ref{lem_liftings_with_prescribed_determinant}, $\sigma|_{G_{K'}}$ lifts to a continuous homomorphism $\tau : G_{K'} \to \GL_2(\F_5)$ such that $\det \tau = \overline{\epsilon}^{-1}$. In particular, there is a character $\overline{\psi} : G_{K'} \to \overline{\F}_5^\times$ such that $\tau = \rhobar|_{G_{K'}} \otimes \overline{\psi}$. Let $\psi$ denote the Teichm\"uller lift of $\overline{\psi}$; then the determinant of $\rho|_{G_{K'}} \otimes \psi$ equals $\epsilon^{-1}$.
\begin{lemma}
The representation $\tau$ satisfies the following conditions:
\begin{enumerate}
\item $\tau|_{G_{K'(\zeta_5)}}$ is absolutely irreducible and $\tau$ is non-exceptional.
\item $\tau$ is decomposed generic. 
\end{enumerate}
\end{lemma}
\begin{proof}
The representation $\tau|_{G_{K'(\zeta_5)}}$ is absolutely irreducible because its projective image contains $\PSL_2(\bbF_5)$. If $\zeta_5 \in K'$ then $\sqrt{5} \in K$ and so $K' = K(\Delta \circ \sigma) = K(\zeta_5)$; this possibility is ruled out because $\sigma$ is non-exceptional. It follows that $\tau$ is non-exceptional. The representation $\tau$ is decomposed generic  because $\rhobar|_{G_{K'}}$ is (and this condition only depends on the associated projective representation).
\end{proof}
Thanks to the lemma, we can apply \cite[Lemma 9.7]{AKT} and \cite[Corollary 9.13]{AKT} to conclude the existence of a modular elliptic curve $E$ over $K'$ such that $\rhobar_{E, 5} \cong \tau$ and for each place $v | 5$ of $K'$, $E$ has multiplicative reduction at the place $v$. We can then apply the automorphy lifting theorem \cite[Theorem 8.1]{AKT} to conclude that $\rho|_{G_{K'}} \otimes \psi$ is automorphic, hence that $\rho|_{G_{K'}}$ is automorphic. It follows by cyclic descent \cite{L} that $\rho$ and hence $\sigma$ are also automorphic, and this completes the proof.
\end{proof}

\bibliographystyle{alpha}
\bibliography{ModLiftBib}

\end{document}